\documentclass{article}
\usepackage{graphicx, hyperref, relsize}
\usepackage{pdflscape, multirow, tabularray} 
\usepackage[english]{babel}
\usepackage{authblk, fullpage, microtype, amsmath, amsthm, enumerate, amsfonts, amssymb, hyperref, xcolor, mathtools}

\newtheorem{manualtheoreminner}{Theorem}
\newenvironment{manualtheorem}[1]{%
  \IfBlankTF{#1}
    {\renewcommand{\themanualtheoreminner}{\unskip}}
    {\renewcommand\themanualtheoreminner{#1}}%
  \manualtheoreminner
}{\endmanualtheoreminner}

\newtheorem{manualpropinner}{Proposition}
\newenvironment{manualproposition}[1]{%
  \IfBlankTF{#1}
    {\renewcommand{\themanualpropinner}{\unskip}}
    {\renewcommand\themanualpropinner{#1}}%
  \manualpropinner
}{\endmanualpropinner}

\newcommand{\R}{\mathbb{R}}
\newcommand{\Q}{\mathbb{Q}}
\newcommand{\Z}{\mathbb{Z}}

\theoremstyle{definition}

\newcommand{\F}{\mathbb{F}}

\renewcommand{\P}{\mathbb{P}}

\newcommand{\Jac}{\textnormal{Jac}}

\newtheorem{theorem}{Theorem}[section]
\newtheorem{lemma}[theorem]{Lemma}
\newtheorem{prop}{Proposition}

\newtheorem{conj}{Conjecture}
\newtheorem{definition}{Definition}
\newtheorem{example}{Example}
\newtheorem{remark}{Remark}

\DeclareMathOperator{\Rank}{Rank}

\title{Rational Points on a Family of Genus 3 Hyperelliptic Curves}
\author{Roberto Hernandez}
\affil{University of California, Irvine}
\date{\vspace{-5ex}}

\begin{document}

\maketitle

\begin{abstract}
\noindent We compute the rational points on certain members of the following family of hyperelliptic curves
\begin{eqnarray*}
    C_a \colon y^2 = x^8 + (4-4a^4) x^6 + (8a^4 + 6)x^4 + (4-4a^4)x^2 + 1
\end{eqnarray*}
via the method first developed by Dem'yanenko \cite{dem1966rational} and then further generalized by Manin \cite{manin1969p}. In particular, we show that the method of Chabauty--Coleman, while applicable to certain members of this family, is not the most efficient way of computing $C_a(\Q)$. We adapt the approach of \cite{kulesz1999application}, incorporating root numbers to further restrict the possible ranks of the elliptic curves arising in the Jacobian decomposition.
\end{abstract}


\section{Introduction}
In 1983, Faltings' \cite{faltings1983endlichkeitssatze} proved Mordell's Conjecture stating that any non-singular curve $C$ defined over a number field of genus greater than one has finitely many rational points. Since then, there have been simplifications to his original argument, most notably done through the work of Bombieri \cite{bombieri1990mordell} and Vojta \cite{vojta1991siegel}. While Faltings' theorem marked a milestone in arithmetic geometry, its limitation is that it does not give an instructive way to obtain all the rational points on $C$. Effectively finding rational points on curves of genus greater than one is an active area of research, with many working to better understand or improve current methods. In practice, the most powerful tool for computing rational points on curves has been the Chabauty--Coleman method. In this work, we focus on the method developed by Dem'yanenko \cite{dem1966rational} in 1968, extended by Manin \cite{manin1969p} in 1969. We present an explicit family of curves 
\begin{eqnarray*}
    C_a \colon y^2 = x^8 + (4-4a^4) x^6 + (8a^4 + 6)x^4 + (4-4a^4)x^2 + 1,
\end{eqnarray*}
for which the method of Dem'yanenko--Manin proves particularly effective. In this setting, the method of Chabauty--Coleman does not generally apply in a useful way. In particular, the Jacobian of $C_a$ is isogenous to a product of elliptic curves, written $\Jac(C_a) \sim E_a \times E_a \times E'$, and Chabauty--Coleman can only be applied in the special cases where $\Rank(E_a) = 0$ or $\Rank(E') = 0$. In those cases, however, the rational points on $C_a$ can already be determined by pulling back the (finite) rational points on either $E_a$ or $E'$. The more interesting case is where $\Rank(E_a) = 1$ and $\Rank(E') \neq 0$, because the conditions to apply Chabauty--Coleman are not met, and the Dem'yanenko--Manin method becomes essential. Concretely, we will show that
\begin{eqnarray*}
    C_{237}(\Q) = \{(\pm1 : \pm 4: 1), (0: \pm 1: 1), (1: \pm1 : 0)\},
\end{eqnarray*}
to showcase a particular instance of the method. We will also along the way prove certain facts about a family of elliptic curves which arise in our construction. In particular, we prove the following results about the Mordell-Weil group of the particular elliptic curve $E_a \colon y^2 = x^3 + 2a^2x^2 + (a^4-4)x$.
\begin{manualproposition}{1}
    Suppose that $a \equiv 0 \textnormal{ (mod 3)}$, then $E_a(\Q)_\textnormal{tor} \cong (\Z/2\Z)^2$.
\end{manualproposition}

\begin{manualtheorem}{4.1}
    Suppose that $a^2-2$ and $a^2 + 2$ are simultaneously prime, then $\Rank(E_a) = 1$.
\end{manualtheorem}

\noindent We recall that that the method of Chabauty-Coleman uses $p$-adic techniques to produce $p$-adic locally analytic functions, and relies on the assumption that rank of the Jacobian of $C$ over $\Q$, say $r$, is strictly smaller than the genus of $C$. In 1941, Chabauty \cite{chabauty1941points} proved that if the rank of the Jacobian is smaller than $g$, then the set $C(\Q_p) \cap \overline{\Jac(\Q)}$ is finite. Moreover, in 1985, Coleman \cite{coleman1985} showed that if you also assume that you have a prime $p > 2g$ of good reduction, then you can place an upper bound on the number of rational points on $C$.

\begin{manualtheorem}{}[Chabauty-Coleman]
    Let $C$ be a smooth, projective, geometrically integral curve over $\Q$ and let $r \coloneq \Rank \Jac_C(\Q)$. If $r < g$ and $p > 2g$ is a prime of good reduction, then $C(\Q_p) \cap \overline{\Jac_C(\Q)}$ is finite and $\#C(\Q) \leq \#C(\F_p) + 2g - 2$.
\end{manualtheorem}

\noindent We highlight a family of curves for which we cannot apply this theorem because $\Rank(\Jac_{C_a}(\Q)) \ge g$, but the method of Dem'yanenko--Manin gives a procedure to find $C_a(\Q)$. For that, we will force the Jacobian of the curves to fully decompose, and then aim to restrict the ranks of the elliptic curves which appear in the decomposition.

\subsubsection*{Code}
\noindent Throughout this project, we used the computer algebra system Magma \cite{bosma1997magma}. All calculations and claims can be found at the following GitHub repository: 

\begin{center}
    \href{https://github.com/robertohernandez8/rational-points-genus-3-hyperelliptic-curves.git}{https://github.com/robertohernandez8/rational-points-genus-3-hyperelliptic-curves.git}
\end{center} 

\section{Background}
\noindent First, we provide some relevant background material which can be found in more detail for instance in \cite{hindry2013diophantine}. Everything we discuss in this article will be over $\Q$, unless stated otherwise. 

\noindent An elliptic curve is an abelian variety of dimension one, equipped with a distinguished point denoted $\mathcal{O} = [0:1:0]$. It can be described concretely via its Weierstrass equation
\begin{equation*}
    E \colon y^2 + a_1 xy + a_3 y = x^3 + a_2 x^2 + a_4 x + a_6. 
\end{equation*}

\noindent Two important invariants of an elliptic curve are its discriminant, $\Delta$, and the $j$-invariant. We recall some of the invariants associated to elliptic curves below.

\[
\begin{aligned}
b_2 &= a_1^2 + 4a_2, \\
b_4 &= 2a_4 + a_1 a_3, \\
b_6 &= a_3^2 + 4a_6, \\
b_8 &= a_1^2 a_6 + 4a_2 a_6 - a_1 a_3 a_4 + a_2 a_3^2 - a_4^2, \\
c_4 &= b_2^2 - 24 b_4, \\
c_6 &= -b_2^3 + 36 b_2 b_4 - 216 b_6, \\
\Delta &= -b_2^2 b_8 - 8 b_4^3 - 27 b_6^2 + 9 b_2 b_4 b_6, \\
j(E) &= \frac{c_4^{\,3}}{\Delta}.
\end{aligned}
\]

\noindent In particular, the $j$-invariant controls the isomorphism class of an elliptic curve over $\overline{\Q}$, while the discriminant helps detect primes of bad reduction. We also recall that a hyperelliptic curve over $\Q$ is a smooth, projective curve of genus $g \geq 2$ defined over $\Q$ which admits a degree $2$ map to $\P^1$. Equivalently, over $\Q$, we can choose a model in affine coordinates of the form
$$C \colon y^2 = f(x),$$
where $f(x)$ is a monic polynomial of degree $2g + 1$ or $2g + 2$.

\noindent Let $P \in \P^n(\Q)$ and write $P = (x_0, x_1, \dots, x_n)$ with $x_i \in \Z$ and $\gcd(x_0, x_1, \dots, x_n) = 1$. We define the height of $P$ to be
$$H(P) = \max\{|x_0|, |x_1|, \dots, |x_n|\}.$$
The set 
$$\{P \in \P^n(\Q) : H(P) \leq B\}$$
is finite, since there are only finitely many integers $x$ satisfying $|x| \leq B$. 
\begin{definition}\label{multiplicativeheight}
    Let $k$ be a number field and let $P \in \P^n(k)$. The multiplicative height of $P$ is the quantity
    $$H_k(P) = \prod_{v \in M_k} \max\{||x_0||_v, ||x_1||_v, \dots, ||x_n||_v\}$$
    and the logarithmic height is defined as
    $$h_k(P) = \log H_k(P).$$
\end{definition}

\begin{theorem}\label{morphismheight}
    Let $X$ be a closed subvariety of $\P^n$ and let $\varphi \colon X \to \P^n$ be a morphism of degree $d$. Then
    $$h(\varphi(P)) = dh(P) + O(1) \hspace{.5in} \text{ for all } P \in X(\overline{\Q}).$$
\end{theorem}
\noindent It is important to note here that the $O(1)$ depends only on the map $\varphi$, but it independent of the point $P$.

\begin{definition}\label{heightonV}
    Let $\varphi \colon V \to \P^n$ be a morphism. The height on $V$ relative to $\varphi$ is the function
    $$h_\varphi \colon V(\Q) \to [0, \infty), \hspace{.5in} h_\varphi(P) = h(\varphi(P)),$$
    where $h \colon \P^n(\overline{\Q}) \to [0,\infty)$ is the height function on projective space that we defined in the previous section. 
\end{definition}
\noindent In this way, we see that defining heights on varieties is very natural once you have an embedding of your variety into projective space. Recall that any nonconstant function $f$ in the function field $K(E)$ determines a surjective morphism from $E$ to $\P^1$. Concretely, $f$ gives the map
$$f \colon E \to \P^1, \hspace{.5in} P \mapsto \begin{cases}
    [1,0] \text{ if } P \text{ is a pole of } f \\
    [f(P),1] \text{ otherwise. }  
\end{cases}$$

\begin{definition}\label{heightonE}
    Let $E/\Q$ be an elliptic curve and let $f$ be a function in the function field of $E$. The height on $E$ relative to $f$ is the function
    $$h_f \colon E(\overline{\Q}) \to \R, \hspace{.5in} h_f(P) = h(f(P)).$$
\end{definition}

\noindent We are now ready to formally state the main result we use in our exploration, and explain how the method simplifies in the case of rank 1. Let $C$ be a smooth projective curve of genus $g \geq 2$ over a number field $K$. Let $A$ be an abelian variety over $K$, and let $f_1, \dots, f_m$ be morphisms from $C$ to $A$ defined over $K$. Recall that a morphism is a rational map with no base points. 
\begin{definition}
    Let $f_1, \dots, f_m$ be morphisms from $C$ to $A$. We say the $f_i$ are independent if $\sum n_i f_i$ is the constant map, then $n_i = 0$ for all $i$. 
\end{definition}
\noindent In practice, we do not typically use the definition to check whether a set of morphisms are independent, instead we use a lemma due to Cassels \cite{cassels1968theorem}.
\begin{lemma}[Cassels]\label{casselslemma}
    Let $f_1, \dots, f_m$ be morphisms from $C$ to $A$. Let $D$ be the matrix whose coefficients are
    $$\langle f_i, f_j \rangle = \frac{1}{2} \left( d(f_i + f_j) - d(f_i) - d(f_j) \right),$$
    where $d(f_i)$ denotes the degree of the morphism. The $m$ morphisms are independent if and only if $\det(D) \neq 0$.
\end{lemma}
\noindent For our purposes, we will work with explicit morphisms between our hyperelliptic curve $C_a$ and an elliptic curve, $E_a$, and thus these degrees are computable through the use of the computer algebra system Magma. 
\begin{theorem}[Dem'yanenko--Manin]
    Let $C$ be a curve over a number field $K$ and $A$ be an abelian variety over $K$. Let $f_1, \dots, f_m$ be morphisms from $C$ to $A$ that are defined over $K$ and independent. Assume further that $\Rank(A) < m$. Then $C(K)$ is finite and it is possible to effectively bound the height of the points in $C(K)$.
\end{theorem}

\noindent For a proof of the statement and some examples carried out in practice see \cite{girard2005computation}, \cite{kulesz1999application}. The method of Dem'yanenko--Manin simplifies when the $\Rank(A) = 1$ since that would mean that we would only need two independent maps, say $f_1, f_2 \colon C \to A$. Moreover, if we let $P \in C(\Q)$ and $R$ be a generator for the free part of $A(\Q)$, then
$$\begin{cases}
    f_1(P) = [n]R + T_1 \\
    f_2(P) = [m]R + T_2
\end{cases}$$
where $T_i \in A(\Q)_{\text{tor}}$. Taking canonical heights we have
$$\begin{cases}
    \hat{h}(f_1(P)) = n^2\hat{h}(R) \\
    \hat{h}(f_2(P)) = m^2\hat{h}(R).
\end{cases}$$
Subtracting the second equation from the first and taking absolute values gives us the following relation:
$$ \left|\hat{h}(f_1(P)) - \hat{h}(f_2(P))\right| = \left|n^2\hat{h}(R) - m^2\hat{h}(R)\right| = \left|n^2 - m^2\right| \hat{h}(R). $$
If it happens to be the case that $n = \pm m$, then that would imply that $f_1(P) \pm f_2(P) \in A(\Q)_{\text{tor}}$, so we could find those points directly through computation since the torsion group is finite. Note also that we can bound this quantity from above using the triangle inequality. Indeed,
$$ \left|\hat{h}(f_1(P)) - \hat{h}(f_2(P))\right| \leq \left|\hat{h}(f_1(P)) - h(f_1(P))\right| + \left|\hat{h}(f_2(P)) - h(f_2(P))\right| + \Big|h(f_1(P)) - h(f_2(P))\Big|.$$
Thus, if $n \neq \pm m$, then combining these two expressions we can obtain an upper bound, say $B$, such that $|n^2 - m^2| < B$. It is important to note here that the constant $B$ would only depend on $f_1, f_2$, and $A$, but would be independent of $P$. In this way, we bound the possible values of $n$ and $m$, so all that would be left to do would be to find the points on $C(\Q)$ whose images on $A(\Q)$ are of the form $nR + T_i$ with $n < B$. Again, this is a finite computation that can be done using Magma.

\section{Construction of $C_a$}
The main goal of this section is to demonstrate how we obtain the family of genus 3 hyperelliptic curves, and the two independent maps to the same elliptic curve. In particular, we will arrive at the following result.
\begin{manualtheorem}{3.2}\label{3.2}
    Let $a$ be a nonzero integer and let $C_a$ be the genus 3 hyperelliptic curve
\begin{eqnarray*}
    C_a \colon y^2 = x^8 + (4-4a^4)x^6 + (8a^4 + 6)x^4 + (4-4a^4)x^2 + 1.
\end{eqnarray*}
Then, $\Jac(C_a) \sim E_a \times E_a \times E'$.
\end{manualtheorem}

\noindent Consider the quartic $g(x) = x^4 - a^2 x^2 + 1$. Notice that we can pick two distinct non-zero points $z \neq w$ such that $g(z) = g(w)$. In this way we obtain
\begin{eqnarray*}
    z^4 - a^2 z^2 + 1 &=& w^4 - a^2 w^2 + 1 \\
    z^4 - w^4 &=& a^2(z^2 - w^2) \\
    z^2 + w^2 &=& a^2,
\end{eqnarray*}
but this last equation is a circle of radius $a$ in $z$ and $w$. We can parameterize rational points on this circle by considering the line $w = zx + a$, as long as $a \neq 0$. In doing so, we have
\begin{eqnarray*}
    z^2 + (zx + a)^2 &=& a^2 \\
    z^2 + z^2x^2 + 2zxa + a^2 &=& a^2 \\
    z^2(1+x^2) &=& -2zxa \\
    z &=& -\dfrac{2xa}{x^2 + 1}.
\end{eqnarray*}
Thus, we obtain the parameterization of rational points on the circle given by 
$$\left(z(x), w(x) \right) = \left(-\dfrac{2ax}{x^2 + 1}, -\dfrac{a(x+1)(x-1)}{x^2 + 1} \right),$$ 
for any $x \in \Q$. Now, we take the function $z(x) = -\dfrac{2ax}{x^2 + 1}$ and plug it into our quartic $g(x)$. Indeed,
\begin{eqnarray*}
    g(z(x)) = \left(-\dfrac{2ax}{x^2 + 1} \right)^4 - a^2 \left(-\dfrac{2ax}{x^2 + 1} \right)^2 + 1,
\end{eqnarray*}
which after multiplying the entire equation by $(x^2 + 1)^4$ to clear denominators simplifies to 
\begin{eqnarray*}
    g(z(x))(x^2 + 1)^4 &=& x^8 + (4-4a^4)x^6 + (8a^4 + 6)x^4 + (4-4a^4)x^2 + 1.
\end{eqnarray*}
The right side of the above equation will define the $f(x)$ for our hyperelliptic curve. Concretely, our family of genus 3 hyperelliptic curves will be defined as:
\begin{eqnarray*}
    C_a \colon y^2 = x^8 + (4-4a^4)x^6 + (8a^4 + 6)x^4 + (4-4a^4)x^2 + 1.
\end{eqnarray*}
Also note that this construction arose from the first quartic $g(x) = x^4 - a^2x^2 + 1$, and if we define the curve $H_a: y^2 = x^4 - a^2x^2 + 1$, this is a genus 1 hyperelliptic curve. This implies that $H_a$ is in fact an elliptic curve and Magma can give us the Weierstrass equation and the map between the curves. In fact, we have that 
\begin{eqnarray*}
    H_a \colon y^2 = x^4 - a^2x^2 + 1 &\cong& E_a \colon y^2 = x^3 + 2a^2x^2 + (a^4-4)x \\
    (x,y) &\mapsto& (2x^2 - 2y - 1, 4x^3 - 4xy - 2x).
\end{eqnarray*}
In particular, we have constructed a curve with multiple maps to the same elliptic curve, lending itself nicely to the method of Dem'yanenko--Manin. We can in fact be very concrete and display the maps we obtain. Firstly, we obtain the maps 
\begin{eqnarray*}
    \psi_1 \colon C_a &\to& H_a \\
    (x,y) &\mapsto& \left(z(x), \dfrac{y}{(x^2+1)^2}\right) \\
    \psi_2 \colon C_a &\to& H_a \\
    (x,y) &\mapsto& \left(w(x), \dfrac{y}{(x^2+1)^2}\right)
\end{eqnarray*}
and composing $\psi_1, \psi_2$ with the isomorphism between $H_a$ and $E_a$ above, we have two maps from $C_a$ to $E_a$. From here on out, we will denote by the composition of these maps $\varphi_1$ and $\varphi_2$. 
\begin{lemma}
    The maps $\varphi_1, \varphi_2 \colon C_a \to E_a$ are independent.
\end{lemma}
\begin{proof}
    The maps $\varphi_1$ and $\varphi_2$ in projective coordinates are given explicitly by
\begin{align*}
    \varphi_1 &= [-a^2x^6 + 5a^2 x^4z^2 - 2x^2y + 5a^2x^2 z^4 - 2yz^2 - a^2 z^6, \\
    &4a^3x^5z - 24a^3 x^3z^3 +8axyz + 4a^3xz^5, \\
    &x^6 + 3x^4z^2 + 3x^2z^4 + z^6] \\
    \varphi_2 &= [a^2x^6 - 5a^2 x^4z^2 - 2x^2y - 5a^2x^2 z^4 - 2yz^2 + a^2 z^6, \\
    &-2a^3x^6 + 14a^3 x^4z^2 + 4ax^2y - 14a^3 x^2z^4 - 4ayz^2 + 2a^3z^6, \\
    &x^6 + 3x^4z^2 + 3x^2z^4 + z^6].
\end{align*}
The codomain of these maps is an elliptic curve so when we consider $\varphi_1 + \varphi_2$ we need to remember that the addition is defined as on an elliptic curve. A quick calculation on Magma shows that $\deg(\varphi_1) = \deg(\varphi_2) = 2,$ and that $\deg(\varphi_1 + \varphi_2) = 4$. Thus, we obtain the matrix
$$M = \begin{bmatrix}
    2 & 0 \\ 0 & 2
\end{bmatrix}$$
and we can clearly see that $\det(M) \neq 0$ and thus by Lemma \ref{casselslemma} the maps are independent.
\end{proof}
\noindent Now, maps between curves induce maps on their Jacobians so in particular we have maps from $\Jac(C_a) \to \Jac(E_a) \cong E_a$. More importantly, these maps will be surjective and have finite kernel, so they will be isogenies. This means that $\Jac(C_a)$ has two factors of $E_a$ in its isogeny decomposition. We also know that $C_a$ has genus 3, and $\dim \Jac(C_a) = 3$, hence the isogeny decomposition of $\Jac(C_a)$ is $E_a \times E_a \times E'$ where $E'$ is some other elliptic curve. For our purposes, it will be useful to know an equation for $E'$ since we will want to ensure it has nonzero rank so that we avoid Chabauty--Coleman. One way to find an equation for $E'$ is by quotienting $C_a$ by  its automorphisms and searching for an elliptic curve which is non-isogenous to $E_a$. This a finite search that can be done on Magma and we obtain 
\begin{equation*}
    E' \colon y^2 = x^3 + (-4-a^4)x^2 + 4a^4 x,
\end{equation*}
with the automorphism $\tau: C_a \to C_a$ defined via $[x,y,z] \mapsto [z,y,-x]$.
\noindent This discussion completes the proof for Theorem \ref{3.2}. We are now in the position where we have a family of genus 3 hyperelliptic curves whose Jacobian decomposes in a way favorable to the use of Dem'yanenko--Manin.
\noindent Our goal now is to find conditions on the parameter $a$ so that $\Rank(E_a) = 1$ (since we have two independent maps $\varphi_1, \varphi_2$) and $\Rank(E') \neq 0$.

\section{Controlling the ranks of $E_a$ and $E'$}
In this section we will show that under certain conditions on $a$, $E_a$ has full 2-torsion, and that $\Rank(E_a) = 1$. We will also make use of the parity conjecture to force the rank of $E'$ to be odd, hence nonzero. A quick search for points on $E_a$ yields the point of infinite order $P = (-a^2,2a) \in E_a(\Q)$ as long as $a \neq 1$ so in this case we know that the rank is bounded below by 1. Our aim is now to show that 1 is also an upper bound for the rank. First, note that 
\begin{eqnarray*}
    E_a \colon y^2 &=& x^3 + 2a^2x^2 + (a^4-4)x \\
    &=& x(x+a^2-2)(x+a^2+2)
\end{eqnarray*} 
so we know that we have three distinct 2-torsion points $(0,0)$, $(-a^2+2,0)$, and $(-a^2-2,0)$. To ease notation, we will let $e_1 = 0$, $e_2 = -a^2 + 2$, and $e_3 = -a^2 - 2$. 

\begin{prop}
    Suppose that $a \equiv 0 \textnormal{ (mod 3)}$, then $E_a(\Q)_\textnormal{tor} \cong (\Z/2\Z)^2$.
\end{prop}

\begin{proof}
    First note that since $a \equiv 0 \textnormal{ (mod 3)}$, then $a^2 \equiv 0 \textnormal{ (mod 3)}$ and so $a^2 - 2 \equiv 1 \textnormal{ (mod 3)}$ and $a^2 + 2 \equiv 2 \textnormal{ (mod 3)}$. Thus, $3$ never divides $a^2-2$ or $a^2+2$. Per our above discussion, we already know that $E_a$ contains full 2-torsion. Further, since $E_a$ has discriminant
    \begin{eqnarray*}
        \Delta(E_a) = 2^8(a^2-2)^2(a^2+2)^2,
    \end{eqnarray*}
    it has good reduction at the prime 3. Reducing modulo 3 we see that 
    \begin{eqnarray*}
        \tilde{E_a} \colon y^2 = x^3 + 2x
    \end{eqnarray*}
    which one can easily verify has $\#\tilde{E_a}(\F_3) = 4$. Now, since $E_a[2] \subset E_a(\Q)_\textnormal{tor}$ and $E_a(\Q)_\textnormal{tor}$ injects into $\tilde{E_a}(\F_3)$, it must be that $E_a[2] = E_a(\Q)_\textnormal{tor}$.
\end{proof}

\begin{theorem}
    Suppose that $a^2-2$ and $a^2 + 2$ are simultaneously prime, then $\Rank(E_a) = 1$. \label{rankE_a}
\end{theorem}

\begin{proof}
We proceed via a complete 2-descent as in \cite{silverman2009arithmetic} (X.1.4). The elliptic curve $E_a$ has discriminant $\Delta(E_a) = 2^8(a^2-2)^2(a^2+2)^2$ and so the set $S = \{2, a^2-2, a^2+2, \infty\}$. We now consider the set 
\begin{eqnarray*}
    \Q(S,2) = \{ b \in \Q^*/(\Q^*)^2 \colon \textnormal{ord}_p(b) \equiv 0 (\textnormal{ mod } 2) \textnormal{ for all } p \notin S\}.
\end{eqnarray*}
Concretely, we can identify the set $\Q(S,2)$ with 
\begin{eqnarray*}
    \{ \pm 1, \pm 2, \pm (a^2-2), \pm (a^2+2), \pm 2(a^2-2), \pm 2(a^2+2), \pm (a^4-4), \pm2(a^4-4) \}.
\end{eqnarray*}
There is an injective homomorphism $E_a(\Q)/2E_a(\Q) \hookrightarrow \Q(S,2) \times \Q(S,2)$ defined by
\begin{eqnarray*}
    P = (x,y) \mapsto \begin{cases}
        (x - e_1, x- e_2) & \textnormal{ if } x \neq e_1, e_2 \\
        \left(\dfrac{e_1 - e_3}{e_1 - e_2}, e_1 - e_2 \right) & \textnormal{ if } x = e_1 \\
        \left( e_2 - e_1, \dfrac{e_2 - e_3}{e_2 - e_1} \right) & \textnormal{ if } x = e_2 \\
        (1,1) & \textnormal{ if } P = \mathcal{O}.
    \end{cases}
\end{eqnarray*}
The pair $(b_1,b_2) \in \Q(S,2) \times \Q(S,2)$ will be the image of a point $P = (x,y) \in E_a(\Q)/2E_a(Q)$ if and only if the equations 
\begin{eqnarray}
    b_1 z_1^2 - b_2z_2^2 &=& -a^2 + 2 \\
    b_1z_1^2 - b_1b_2 z_3^2 &=& -a^2 - 2,
\end{eqnarray}
have a solution $(z_1, z_2, z_3) \in \Q^* \times \Q^* \times \Q$. There are 256 pairs in $\Q(S,2) \times \Q(S,2)$ and for each one we must check whether the defining system of equations has a solution in $\Q$. Luckily, the fact that the map $E_a(\Q)/2E_a(\Q) \rightarrow \Q(S,2) \times \Q(S,2)$ is a homomorphism greatly reduces the number of cases we need to check directly. Indeed, if $(b_1,b_2)$ and $(b_1',b_2')$ come from $E_a(\Q)$, then so does $(b_1b_1', b_2b_2')$, while if $(b_1,b_2)$ comes from $E_a(\Q)$ but $(b_1',b_2')$ does not, then $(b_1b_1', b_2b_2')$ will not. We now proceed by describing certain cases for pairs $(b_1,b_2)$ where we can either find a solution to the system or explain their non-existence. Our results will be summarized in Table \ref{Table:1} at the end of the paper where each cell will either have the corresponding point on $E_a(\Q)$ if the system has a solution, or the obstruction of a solution to the system.
\begin{enumerate}
    \item If $b_1 > 0$ and $b_2 < 0$, then $(2)$ has no solutions in $\R$ since the right-hand side of the equation will be negative while the left hand side will be positive.
    \item The four $2$-torsion points $\{\mathcal{O}, (0,0), (-a^2+2,0), (-a^2-2,0)\}$ on $E_a(\Q)$ map via the map described above to the pairs $\{(1,1), ((a^2-2)(a^2+2),a^2-2), (-(a^2-2),-(a^2-2)), (-(a^2+2),-1)\}$, respectively.
    \item Consider the pair $b_1 = -1$ and $b_2 = -2$. Then the system
    \begin{eqnarray*}
        -z_1^2 + 2z_2^2 &=& -a^2 + 2 \\
        -z_1^2 - 2z_3^2 &=& -a^2 - 2,
    \end{eqnarray*}
    has the solution $(a,1,1)$ with the corresponding point $(-a^2,2a) \in E_a(\Q)$.
    \item Taking the pair $(-1,-2)$ from above and multiplying it by the three pairs coming from the non-trivial $2$-torsion points gives three new pairs: $\{(-(a^2-2)(a^2+2), -2(a^2-2)), (a^2-2,2(a^2-2)),(a^2+2,2)\}$ whose corresponding points on $E_a(\Q)$ are: $\left\{ \left(\dfrac{-a^4+4}{a^2}, \dfrac{-2a^4+8}{a^3}\right), (a^2-2,2a^3-4a),(a^2+2,-2a^3-4a)\right\}$.
    \item If $b_1 \not\equiv 0 \textnormal{ (mod } a^2 + 2)$ and $b_2 \equiv 0 \textnormal{ (mod } a^2 + 2)$ then $(1)$ implies that $z_1,z_2 \in \Z_{a^2+2}$. Then from $(2)$ we see that $z_1 \equiv 0 \textnormal{ (mod } a^2+2)$, which would imply that 
    \begin{eqnarray*}
        -a^2 + 2 \equiv 0 \textnormal{ (mod }a^2+2),
    \end{eqnarray*}
    which is not true. Therefore there are no solutions in $\Q_{a^2+2}$.
    \item If $b_1 \not\equiv 0 \textnormal{ (mod } a^2 - 2)$ and $b_2 \equiv 0 \textnormal{ (mod } a^2 - 2)$ then through a similar argument as before, reversing the roles of equations $(1)$ and $(2)$ we obtain that the system does not have any solutions in $\Q_{a^2-2}$.
    \item Consider the pair $b_1 = -1$ and $b_2 = 1$. We obtain the system
    \begin{eqnarray*}
        -z_1^2 - z_2^2 &=& -a^2 + 2 \\
        -z_1^2 + z_3^2 &=& -a^2 - 2.
    \end{eqnarray*}
    Subtracting the second equation from the first gives 
    \begin{eqnarray*}
        -z_2^2 - z_3^2 &=& 4,
    \end{eqnarray*}
    which has no solutions over $\R$. This same line of reasoning also shows that the pairs $(b_1,b_2) = (-1,2)$ and $(b_1,b_2) = (-2,1)$ have no solutions over $\R$.
    \item Consider the pair $(b_1,b_2) = (1,2)$. Then we have the system
    \begin{eqnarray*}
        z_1^2 - 2z_2^2 &=& -a^2 + 2 \\
        z_1^2 - 2z_3^2 &=& -a^2 - 2.
    \end{eqnarray*}
    Subtracting the second equation from the first gives us 
    \begin{eqnarray*}
        2z_3^2 - 2z_2^2 &=& 4 \\
        \implies z_3^2 - z_2^2 &=& 2 \\
        \implies (z_3+z_2)(z_3-z_2) &=& 2.
    \end{eqnarray*}
    Since the right-hand side of the equation is even, $z_3+z_2$ and $z_3-z_2$ must have the same parity. However, they cannot both be odd, otherwise their product would be odd. Also, they cannot both be even, because then their product would be divisible by four. Thus, this equation has no solutions in $\R$.
    \item Consider the pair $(b_1,b_2) = (2,1)$. We obtain the system
    \begin{eqnarray*}
        2z_1^2 - z_2^2 &=& -a^2 + 2 \\
        2z_1^2 - 2z_3^2 &=& -a^2 - 2.
    \end{eqnarray*}
    Subtracting the second equation from the first gives
    \begin{eqnarray*}
        2z_3^2 - z_2^2 = 4,
    \end{eqnarray*}
    and reducing modulo 2 implies that $z_2 \equiv 0 \textnormal{ (mod }2)$. Since $z_2$ is even, the first equation can be reduced modulo 2 as:
    \begin{eqnarray*}
        2z_1^2 - z_2^2 \equiv 0 \equiv a^2 \textnormal{ (mod }2),
    \end{eqnarray*}
    which implies that $a$ is even, which contradicts our initial assumption that $a^2-2$ and $a^2+2$ are simultaneously prime. This same line of reasoning also shows that the pairs $(b_1,b_2) = (2,2)$ and $(b_1,b_2) = (-2,-2)$ also contradict the primality of $a^2-2$ and $a^2+2$. We will denote this obstruction in Table \ref{Table:1} as $\bigstar$.
    \item Finally, we consider the pair $(b_1,b_2)=((a^2-2)(a^2+2),2)$. This gives the system
    \begin{eqnarray*}
        (a^2-2)(a^2+2)z_1^2 - 2z_2^2 &=& -a^2 + 2 \\
        (a^2-2)(a^2+2)z_1^2 - 2(a^2-2)(a^2+2)z_3^2 &=& -a^2 - 2.
    \end{eqnarray*}
    Reducing the second equation modulo $a^2-2$ gives 
    \begin{eqnarray*}
        -(a^2+2) \equiv 0 \textnormal{ (mod }a^2-2),
    \end{eqnarray*}
    but this cannot happen since $a^2-2$ and $a^2+2$ are distinct primes. We will denote this obstruction in Table \ref{Table:1} as $\blacklozenge$.
\end{enumerate}
The remaining cells of Table \ref{Table:1} are now filled in using the multiplicative rule we discussed earlier. The size of $E_a(\Q)/2E_a(\Q)$ is $2^{2+r}$ where $r \coloneq \Rank(E_a)$, and we have found that only $2^3$ pairs in $\Q(S,2) \times \Q(S,2)$ are the image of a point in $E_a(\Q)/2E_a(\Q)$. Therefore, $r = 1$.
\end{proof}

\begin{remark}
    We could try this same analysis for $E'$, however the factorization of its discriminant forces certain primality conditions that would make $\Q(S,2) \times \Q(S,2)$ a set of size 4096. Even then, this provides an upper bound for its rank, and we were not able to find an explicit point of infinite order as we did in the case of $E_a$ to also place a lower bound.
\end{remark}

\begin{definition}
The global root number $w(E/K)$ of an elliptic curve over $K$ is defined as the product of the local root numbers $w(E/K_v) \in \{-1, 1\}$. In particular,
$$ w(E/K) = \prod_v w(E/K_v), $$
where the product runs over all places of $K$, including the infinite ones.
\end{definition}
\noindent Local root numbers of elliptic curves are defined using epsilon-factors of Weil--Deligne representations, however, for our purposes we will not need to introduce formal definitions. We instead direct the interested reader to \cite{rohrlich1993variation} for a more detailed exposition of local root numbers. They have been classified for all places of number fields, so we will not concern ourselves with proving any results, but instead use them for our computations. For instance, we make use of the following result from \cite{kellock2023rootnumbersparityphenomena}. We remark that the result has more cases, and we have only included the ones relevant to our purposes.
\begin{theorem}\label{kellockdokchitser}
Let $E$ be an elliptic curve over a local field $\mathcal{K}$ of characteristic zero. When $\mathcal{K}$ is non-Archimedean, let $k$ be its residue field and let $v \colon \mathcal{K}^{\times} \to \Z$ denote the normalized valuation with respect to $\mathcal{K}$. Let $\left( \dfrac{*}{k} \right)$ denote the quadratic residue symbol on $k^{\times}$ and $(a,b)_\mathcal{K}$ denote the Hilbert symbol in $\mathcal{K}$. 
\begin{enumerate}
    \item If $\mathcal{K}$ is Archimedean, then $w(E/\mathcal{K}) = -1$.
    \item If $E/\mathcal{K}$ has good reduction, then $w(E/\mathcal{K}) = 1$.
    \item If $E/\mathcal{K}$ has split multiplicative reduction, then $w(E/\mathcal{K}) = -1$.
    \item If $E/\mathcal{K}$ has non-split multiplicative reduction, then $w(E/\mathcal{K}) = 1$.
    \item If $E/\mathcal{K}$ has additive, potentially multiplicative reduction and $\text{char}(k) = 2$, then $w(E/\mathcal{K}) = (-1,-c_6)_\mathcal{K}$. In particular, if $\mathcal{K} = \Q_2$, then
    \begin{equation*}
        w(E/\Q_2) = \begin{cases}
            -1 & \text{ if } c_6^{'} \equiv 1 \text{ (mod } 4) \\
            +1 & \text{ if } c_6^{'} \equiv 3 \text{ (mod } 4) 
            \end{cases}
    \end{equation*}
    where $c_6^{'} = \dfrac{c_6}{2^{v(c_6)}}$.
\end{enumerate}
\end{theorem}

\noindent In particular, this result tells us that to determine local root numbers, we only need to determine the reduction type of our elliptic curve at various primes. Now, over $\Q$, the places are precisely the primes, and the usual infinite place. The only places where our elliptic curve may have bad reduction is at the primes which divide the discriminant and so we only need to consider those primes, since the primes of good reduction have local root number 1, and as such will not change the parity of the global root number.
\begin{conj}[Parity Conjecture] \label{parityconjecture}
    Let $E/K$ be an elliptic curve over a number field. Then 
    $$(-1)^{\Rank(E)} = w(E/K)$$
    where $w(E/K)$ is the global root number of $E$.
\end{conj}

\noindent We were not able to find a point of infinite order on $E'$, which means that in order to try and control its rank, we will need to make use of the Parity Conjecture. Note that its discriminant is $\Delta(E') = 2^8 a^8 (a^2 - 2)^2 (a^2 + 2)^2$, which complicates things because recall the fact that $a^2-2$ and $a^2+2$ being simultaneously prime forces $a \equiv 0 \textnormal{ (mod }3)$. Thus, we will need to do an analysis of $E'$ at the primes $2, 3$, $q, a^2 - 2$, and  $a^2 + 2$ for $q$ a prime bigger than 3. We will see in our next result that we will need another condition on $q$, in particular, we will need $q \equiv 3 \text{ mod } 4$.

\begin{theorem}
    Let $a = 3q$, with $q > 3$ prime, $q \equiv 3 \text{ mod } 4$ and $a^2 - 2$, $a^2 + 2$ are both prime. Then, assuming the Parity Conjecture, $\Rank(E')$ is odd.
\end{theorem}
\begin{proof}
    We have $\Delta(E') = 2^8 a^8 (a^2 - 2)^2 (a^2 + 2)^2$ and so we need to find the reduction type at the following primes $p \in \{2, 3, q, a^2 - 2, a^2 + 2 \}$. \\
    First consider $p = 2$. Using Theorem 3.7 from \cite{Barrios_2022} and noting that in their notation, $E' = E_{C_2 \times C_2}(-4, -a^4, 1) = E_{C_2 \times C_2}(a,b,d)$, we have $v_2(a) = 2$, and $bd \equiv 3 \text{ mod } 4$ and thus the reduction type is $I_0^*$. Since we have additive, potentially multiplicative reduction and $c_6' \equiv 3 \text{ mod } 4$, Theorem \ref{kellockdokchitser} tells us that $w_2(E') = 1$. We should note that by Theorem 3.7 in \cite{Barrios_2022}, we know that $p = 2$ is the only prime at which we have additive reduction, and thus for the rest of the primes we only need to check whether the reduction type is split or non-split. \\
    Consider now $p = 3$. Reducing $E'$ modulo $p$ gives
    $$ y^2 \equiv x^3 + 2x^2 \text{ mod } p, $$
    and since $\left( \dfrac{2}{3} \right) = -1$, we have non-split multiplicative reduction. Hence, by Theorem \ref{kellockdokchitser} $w_3(E') = 1$. \\
    Consider now $p = q$. Reducing $E'$ modulo $p$ gives
    $$ y^2 \equiv x^3 - 4x^2 \text{ mod } p,$$
    so $\left( \dfrac{-4}{q} \right)$ will determine the type. Note that
    \begin{eqnarray*}
        \left( \dfrac{-4}{q} \right) &=& \left( \dfrac{-1}{q} \right)\left( \dfrac{4}{q} \right) \\
        &=& \left( \dfrac{-1}{q} \right) \text{ since } q \nmid 4 \\
        &=& -1 \text{ if } q \equiv 3 \text{ mod } 4.
    \end{eqnarray*}
    Thus, by Theorem \ref{kellockdokchitser}, $w_q(E') = 1$. \\
    At the primes $a^2 - 2$ and $a^2 + 2$ we can verify via Magma that the reduction types are split and so at the last two places needed we have $w_p(E') = -1$. \\
    Putting this all together we have $w(E'/\Q) = (-1)(1)(1)(1)(-1)(-1) = -1$, which implies that the rank is odd by the Parity Conjecture.
\end{proof}

\begin{conj} \label{infinitelymanyprimes}
    There are infinitely many $3q$, with $q > 3$ and $q \equiv 3 \text{ mod } 4$ such that $9q^2 - 2$ and $9q^2 + 2$ are simultaneously prime.
\end{conj}

\noindent This conjecture is an instance of Schinzel's Hypothesis H \cite{schinzel1958certaines}, with enough evidence supporting its validity.

\section{Bounding the Difference in Heights}
In this section our aim will be to prove the following inequality concerning the difference between the canonical heights on $E_a$
\begin{eqnarray*}
    \left| \hat{h}_{E_a}(\varphi_1(P)) - \hat{h}_{E_a}(\varphi_2(P)) \right| \leq 51.18 + 2 \log(a).
\end{eqnarray*}
As mentioned before, this will allow us to effectively find all the possible images of the points $P \in C_a(\Q)$ under the maps $\varphi_1$ and $\varphi_2$. Recall that we have explicit maps $\varphi_1$ and $\varphi_2$ and if we are concerned with bounding the quantity $$\Big|\hat{h}_{E_a}(\varphi_1(P)) - \hat{h}_{E_a}(\varphi_2(P))\Big|,$$
then we should attempt to get a better understanding of what equations define $\varphi_1(P)$ and $\varphi_2(P)$ and calculate bounds for their naive height. Recall also that the naive height of a point on an elliptic curve was simply defined as the height of the $x$-coordinate of the image of the point and so we are really only concerned with the $x$-coordinate of these images. In affine coordinates we have
\begin{eqnarray*}
    x(\varphi_1(P)) &=& \dfrac{-a^2 x^4 + 6a^2 x^2 - 2y - a^2}{x^4 + 2x^2 + 1} \\ \\
    x(\varphi_2(P)) &=& \dfrac{a^2 x^4 - 6a^2 x^2 - 2y + a^2}{x^4 + 2x^2 + 1}.
\end{eqnarray*}
We run into an issue here that both numerators in these maps depend on the $y$-coordinate of the point $P$ which means that we will need to use the defining equation of our curve to get the terms purely in terms of $x$. This step is essential for determining lower bounds of the heights of $\varphi_1(P)$ and $\varphi_2(P)$. \\

\noindent Recall that the defining equation for our family of hyperelliptic curves is:
\begin{eqnarray*}
    y^2 &=& x^8 + (4-4a^4) x^6 + (8a^4 + 6) x^4 + (4-4a^4) x^2 + 1 \\
    y &=& \sqrt{x^8 + (4-4a^4) x^6 + (8a^4 + 6) x^4 + (4-4a^4) x^2 + 1}.
\end{eqnarray*}
Note that $4-4a^4$ is negative for all values of $a > 1$ and so, in particular these terms will only decrease the value and so by ignoring them we ``maximize" how large the $y$-coordinate could be. Indeed, we have 
\begin{eqnarray*}
    y' = \sqrt{x^8 + (8a^4 + 6) x^4 + 1} \geq y.
\end{eqnarray*}
Moreover, we can further maximize $y'$ with the following inequality:
\begin{eqnarray*}
    \sqrt{x^8 + (8a^4 + 6) x^4 + 1} &\leq& \sqrt{x^8} + \sqrt{(8a^4 + 6)x^4} + 1 \\
    &\leq& x^4 + (3a^2 + 3)x^4 + 1 = y''.
\end{eqnarray*}
We arrive at the ladder of inequalities $y'' \geq y' \geq y$, and this is how we will minimize the numerator of both the images of $\varphi_1$ and $\varphi_2$. In particular, replacing $y$ with $y''$ in $x(\varphi_1(P))$ and $x(\varphi_2(P))$ we have:
\begin{eqnarray*}
    x(\varphi_1(P)) &=& \dfrac{-a^2 x^4 + 6a^2 x^2 - 2y - a^2}{x^4 + 2x^2 + 1} \\\\
    &\geq& \dfrac{(-2-a^2)x^4 - 6x^2 + (-2-a^2)}{x^4 + 2x^2 + 1} \\\\
    x(\varphi_2(P)) &=& \dfrac{a^2 x^4 - 6a^2 x^2 - 2y + a^2}{x^4 + 2x^2 + 1} \\\\
    &\geq& \dfrac{(a^2-2)x^4 + (-12a^2 -6)x^2 + (a^2 - 2)}{x^4 + 2x^2 + 1}.
\end{eqnarray*}
It will be convenient for us to define the following homogeneous polynomials
\begin{eqnarray*}
    F_1(X,Z) &=& (-2-a^2)X^4 - 6X^2Z^2 + (-2-a^2)Z^4 \\
    G_1(X,Z) &=& X^4 + 2X^2Z^2 + Z^4 \\
    F_2(X,Z) &=& (a^2 - 2)X^4 + (-12a^2 - 6)X^2Z^2 + (a^2 - 2)Z^4 \\
    G_2(X,Z) &=& X^4 + 2X^2Z^2 + Z^4.
\end{eqnarray*}
Note that $F_1(X,1)$ and $G_1(X,1)$ are simply the numerator and denominator for the lower bound of $x(\varphi_1(P))$ and that they are relatively prime in $\Q[X]$, so they generate the unit ideal in $\Q[X]$. This implies that we have the following result.
\begin{lemma} \label{lemma1}
    Define the polynomials 
    \begin{eqnarray*}
        F_1(X,Z) &=& (-2-a^2)X^4 - 6X^2Z^2 + (-2-a^2)Z^4 \\
        G_1(X,Z) &=& X^4 + 2X^2Z^2 + Z^4 \\
        f_1(X,Z) &=& -\dfrac{1}{2(a^2-1)}X^2Z^2 - \dfrac{1}{a^2 - 1} Z^4 \\
        g_1(X,Z) &=& -\dfrac{\frac{1}{2}a^2 + 1}{a^2-1}X^2Z^2 - \dfrac{3}{a^2-1}Z^4 \\
        f_2(X,Z) &=& - \dfrac{1}{2(a^2-1)}X^2Z^2 - \dfrac{1}{a^2-1}X^4 \\
        g_2(X,Z) &=& -\dfrac{\frac{1}{2}a^2 + 1}{a^2-1}X^2Z^2 - \dfrac{3}{a^2-1}X^4.
    \end{eqnarray*}
    Then the following identities hold:
    \begin{eqnarray*}
        F_1(X,Z)f_1(X,Z) + G_1(X,Z)g_1(X,Z) &=& Z^8 \\
        F_1(X,Z)f_2(X,Z) + G_1(X,Z)g_2(X,Z) &=& X^8.
    \end{eqnarray*}
\end{lemma}
\begin{proof}
    The existence of such identities is due to the fact that $F_1(X,Z)$ and $G_1(X,Z)$ are relatively prime homogeneous polynomials. The validity of the identities is a tedious calculation but can verified through computer algebra systems such as Magma, for instance.
\end{proof}
\begin{remark}
    Our inequalities which``removed" the $y$-coordinate from the images of $\varphi_1(P)$ and $\varphi_2(P)$ made it so that we had relatively prime $F_1(X,1)$ and $G_1(X,1)$, which was a necessary step to find our identities.
\end{remark}
\noindent We state our similar result for $F_2(X,Z)$ and $G_2(X,Z)$.
\begin{lemma} \label{lemma2}
    Define the polynomials 
    \begin{eqnarray*}
        F_2(X,Z) &=& (a^2 - 2)X^4 + (-12a^2 - 6)X^2Z^2 + (a^2 - 2)Z^4 \\
        G_2(X,Z) &=& X^4 + 2X^2Z^2 + Z^4 \\
        h_1(X,Z) &=& \dfrac{1}{14a^2 + 2}X^2Z^2 + \dfrac{1}{7a^2 + 1}Z^4 \\
        i_1(X,Z) &=& -\dfrac{\frac{1}{14}a^2 - \frac{1}{7}}{a^2 + \frac{1}{7}}X^2Z^2 + \dfrac{\frac{6}{7}a^2 + \frac{3}{7}}{a^2 + \frac{1}{7}}Z^4 \\
        h_2(X,Z) &=& \dfrac{1}{14a^2 + 2}X^2Z^2 + \dfrac{1}{7a^2 + 1}X^4 \\
        i_2(X,Z) &=& -\dfrac{\frac{1}{14}a^2 - \frac{1}{7}}{a^2 + \frac{1}{7}}X^2Z^2 + \dfrac{\frac{6}{7}a^2 + \frac{3}{7}}{a^2 + \frac{1}{7}}X^4.
    \end{eqnarray*}
    Then the following identities hold:
    \begin{eqnarray*}
        F_2(X,Z)h_1(X,Z) + G_2(X,Z)i_1(X,Z) &=& Z^8 \\
        F_2(X,Z)h_2(X,Z) + G_2(X,Z)i_2(X,Z) &=& X^8.
    \end{eqnarray*}
\end{lemma}
\noindent The polynomials $f_1, f_2, g_1, g_2, h_1, h_2, i_1$ and $i_2$ were all found through the help of Magma using the Euclidean algorithm for polynomials. \\
\noindent If we write $x = x(P) = \dfrac{a}{b}$ in lowest terms then we can write $x(\varphi_1(P)) = \dfrac{F_1(a,b)}{G_1(a,b)}$ and $x(\varphi_2(P)) = \dfrac{F_2(a,b)}{G_2(a,b)}$ as quotients of integers. We can restate the identities for our our lemmas using this notation. For Lemma \ref{lemma1}, we have
\begin{eqnarray*}
    F_1(a,b)f_1(a,b) + G_1(a,b)g_1(a,b) &=& b^8 \\
    F_1(a,b)f_2(a,b) + G_1(a,b)g_2(a,b) &=& a^8.
\end{eqnarray*}
This implies that we have 
\begin{eqnarray*}
    |b^8| &\leq& 2 \max \{|f_1(a,b)|, |g_1(a,b)|\} \max \{|F_1(a,b)|, |G_1(a,b)\} \\
    |a^8| &\leq& 2 \max \{|f_2(a,b)|, |g_2(a,b)|\} \max \{|F_1(a,b)|, |G_1(a,b)\}.
\end{eqnarray*}
Now looking at the expressions of $f_1, f_2, g_1$, and $g_2$, we have the following inequality:
\begin{equation*}
    \max \{|f_1|, |f_2|, |g_1|, |g_2|\} \leq 2 \max \{|a|^4, |b|^4\}.
\end{equation*}
Thus, we can combine the above inequalities to obtain the following lower bound on the height of $\varphi_1(P)$:
\begin{eqnarray*}
    \max \{ |a^8|, |b^8| \} &\leq& 4 \max \{|a|^4, |b|^4\} \max \{|F_1|, |G_1| \} \\
    \frac{\max \{|a|^4, |b|^4 \}}{4} &\leq& \max \{|F_1|, |G_1| \} \\
    \frac{1}{4} H(P)^4 &\leq& H(\varphi_1(P)) \\
    \log\left(\frac{1}{4}\right) + 4h(P) &\leq& h(\varphi_1(P)).
\end{eqnarray*}
We can do the exact same argument for the identities in Lemma \ref{lemma2} and obtain
\begin{eqnarray*}
    \log\left(\frac{1}{4}\right) + 4h(P) &\leq& h(\varphi_2(P)).
\end{eqnarray*}
\noindent Now to obtain the upper bounds for the heights we use $x$-coordinate of the affine maps for $\varphi_1$ and $\varphi_2$ which we recall here:
\begin{eqnarray*}
    x(\varphi_1(P)) &=& \dfrac{-a^2 x^4 + 6a^2 x^2 - 2y - a^2}{x^4 + 2x^2 + 1} \\ \\
    x(\varphi_2(P)) &=& \dfrac{a^2 x^4 - 6a^2 x^2 - 2y + a^2}{x^4 + 2x^2 + 1}.
\end{eqnarray*}
\noindent The first way to obtain an ``easy" upper bound is to simply ignore all the terms which may decrease the numerator, that it all terms which are negative. We thus obtain,
\begin{eqnarray*}
    x(\varphi_1(P)) &=& \dfrac{-a^2 x^4 + 6a^2 x^2 - 2y - a^2}{x^4 + 2x^2 + 1} \\
    &\leq& \dfrac{6a^2x^2}{x^4 + 2x^2 + 1} \\\\
    x(\varphi_2(P)) &=& \dfrac{a^2 x^4 - 6a^2 x^2 - 2y + a^2}{x^4 + 2x^2 + 1} \\
    &\leq& \dfrac{a^2 x^4 + a^2}{x^4 + 2x^2 + 1}.
\end{eqnarray*}
\noindent If we let $x = x(P) = \dfrac{r}{s}$ as before, we can rewrite these as:
\begin{eqnarray*}
    x(\varphi_1(P)) &\leq& \dfrac{6a^2r^2s^2}{r^4 + 2r^2s^2 + s^4} \\\\
    x(\varphi_2(P)) &\leq& \dfrac{a^2 r^4 + s^4a^2}{r^4 + 2r^2s^2 + s^4}.
\end{eqnarray*}
\noindent Now, since the height is simply the max between the numerator and the denominator we only need to determine which of the two is largest. In other words, we have:
\begin{eqnarray*}
    H(\varphi_1(P)) &\leq& \max \{6a^2 H(P)^4, 4H(P)^4\} \\
    h(\varphi_1(P)) &\leq& \log(6) + 2\log(a) + 4h(P)
\end{eqnarray*}
and 
\begin{eqnarray*}
    H(\varphi_2(P)) &\leq& \max \{2a^2 H(P)^4, 4H(P)^4\} \\
    h(\varphi_2(P)) &\leq& \log(2) + 2\log(a) + 4h(P).
\end{eqnarray*}
\noindent We summarize the bounds we have discussed above in the following lemma.
\begin{lemma}\label{boundsonimages}
    For $P \in C_a(\Q)$, we have:
    \begin{eqnarray*}
        \log \left( \frac{1}{4} \right) + 4h(P) \leq h(\varphi_1(P)) \leq \log(6) + 2 \log(a) + 4 h(P) \\
        \log \left( \frac{1}{4} \right) + 4h(P) \leq h(\varphi_2(P)) \leq \log(2) + 2 \log(a) + 4 h(P).
    \end{eqnarray*}
\end{lemma}
\noindent Combining the inequalities from Lemma \ref{boundsonimages}, we obtain the following inequality:
\begin{equation}\label{5.1}
    \left| h(\varphi_1(P)) - h(\varphi_2(P)) \right| \leq \log(6) + 2 \log(a) - \log \left( \frac{1}{4} \right).
\end{equation}
We also have the following lemma regarding the difference between the canonical height and naive height on our elliptic curve $E_a$.
\begin{lemma}\label{silvermanbound}
    Let $P \in E_a(\Q)$. Then we have:
    \begin{equation}\label{5.2}
        \left| \hat{h}_{E_a}(P) - h_{E_a}(P)\right| \leq 24.
    \end{equation}
\end{lemma}
\begin{proof}
    This can be found via the Magma command \begin{verbatim}
        SilvermanBound(E_a);. 
    \end{verbatim} 
\end{proof}

\begin{theorem}\label{canonicalheightbound}
    Let $P \in C_a(\Q)$. Then we have 
    \begin{equation*}
        \left| \hat{h}_{E_a}(\varphi_1(P)) - \hat{h}_{E_a}(\varphi_2(P)) \right| \leq 51.18 + 2 \log(a).
    \end{equation*}
\end{theorem}
\begin{proof}
    This is a combination of the inequalities (\ref{5.1}) and \ref{5.2} from the above lemmas.
\end{proof}

\begin{example}
    Let $a = 237$. Note that this is of the form required by the conditions in Conjecture \ref{infinitelymanyprimes}. By Theorem \ref{canonicalheightbound}, we have that
    \begin{equation*}
        \left| \hat{h}_{E_{237}}(\varphi_1(P)) - \hat{h}_{E_{237}}(\varphi_2(P)) \right| \leq 62.1.
    \end{equation*}
    Also, since $a = 237$, we have 
    \begin{equation*}
        E_{237} : y^2 = x^3 + 112338x^2 + 3154956557x
    \end{equation*}
    and using Magma we can find a generator for its free part, say $R$. We find $\hat{h}(R) = 5.29$. This implies that 
    \begin{equation*}
        \left|n^2 - m^2\right| \leq 11.73 \Rightarrow \max \{ |n|, |m| \} \leq \frac{1}{2} (11.73 + 1) < 7.
        \end{equation*}
        Now all that is left to do is perform a search for either: 
    \begin{itemize}
        \item $P \in C_{237}(\Q)$ such that $\varphi_1(P) \pm \varphi_2(P) \in E_{237}(\Q)_{\text{tor}}$
        \item $P \in C_{237}(\Q)$ such that $\varphi_1(P) = nR + T$ or $\varphi_2(P) = nR + T$ with $T \in E_{237}(\Q)_{\text{tor}}$ and $|n| \leq 6$.
    \end{itemize}
    Even with the large coefficients on the equations, these calculations were carried out on Magma's online calculator in under 3 seconds and we find:
    \begin{equation*}
        C_{237}(\Q) = \{(\pm 1: \pm 4: 1), (0: \pm 1: 1), (1: \pm 1: 0) \}.
    \end{equation*}
\end{example}

\begin{remark}
    When considering the values of $a \in [1,1000]$, we found that 552 of them gave $\Rank(E_a) = 1$ and $\Rank(E') > 0$, so the method of Dem'yanenko--Manin can be used for many curves in this family in an automated way.
\end{remark}

\subsubsection*{Acknowledgements} This work is part of the authors Ph.D. thesis. We would like to thank David Zureick-Brown for his introduction to the area and helpful conversations throughout graduate school. We would also like to thank Nathan Kaplan for his comments on an earlier version of this paper.

\clearpage
\thispagestyle{empty}

\begin{table}
\centering
\caption{Values corresponding to Theorem \ref{rankE_a}.}
\label{Table:1}
\rotatebox{90}{%
\scalebox{.55}[.85]{
\hspace{5in}\begin{tblr}{
  cell{10}{2} = {c=8,r=8}{},
  columns={halign=c},
  vlines,
  hline{1-10,18} = {-}{},
  hline{11-17} = {1,10-17}{},
}
$b_2$ \textbackslash $b_1$ & 1 & 2 & $a^2 - 2$ & $a^2 + 2$ & $2(a^2-2)$ & $2(a^2+2)$ & $(a^2-2)(a^2+2)$ & $2(a^2-2)(a^2+2)$ & $-1$ & $-2$ & $-(a^2 - 2)$ & $-(a^2 + 2)$ & $-2(a^2-2)$ & $-2(a^2+2)$ & $-(a^2-2)(a^2+2)$ & $-2(a^2-2)(a^2+2)$ \\ 
1 & $\mathcal{O}$ & $\bigstar$ & $\Q_{a^2-2}$ & $\R$ & $\Q_{a^2-2}$ & $\bigstar$ & $\Q_{a^2-2}$ & $\Q_{a^2-2}$ & $\R$ & $\R$ & $\Q_{a^2-2}$ & $\R$ & $\Q_{a^2-2}$ & $\bigstar$ & $\Q_{a^2-2}$ & $\Q_{a^2-2}$ \\
2 & $\R$ & $\bigstar$ & $\Q_{a^2-2}$ & $(a^2+2, -2a^3 - 4a)$ & $\Q_{a^2-2}$ & $\bigstar$ & $\Q_{a^2-2}$ & $\Q_{a^2-2}$ & $\R$ & $\bigstar$ & $\Q_{a^2-2}$ & $\R$ & $\Q_{a^2-2}$ & $\R$ & $\blacklozenge$ & $\Q_{a^2-2}$ \\
$a^2 - 2$ & $\Q_{a^2-2}$ & $\Q_{a^2-2}$ & $\Q_{a^2-2}$ & $\Q_{a^2-2}$ & $\bigstar$ & $\Q_{a^2-2}$ & $(0,0)$ & $\bigstar$ & $\Q_{a^2-2}$ & $\Q_{a^2-2}$ & $\R$ & $\Q_{a^2-2}$ & $\bigstar$ & $\Q_{a^2-2}$ & $\R$ & $\R$ \\
$(a^2 + 2)$ & $\Q_{a^2+2}$ & $\Q_{a^2+2}$ & $\Q_{a^2+2}$ & $\Q_{a^2+2}$ & $\Q_{a^2+2}$ & $\Q_{a^2+2}$ & $\Q_{a^2+2}$ & $\Q_{a^2+2}$ & $\Q_{a^2+2}$ & $\Q_{a^2+2}$ & $\Q_{a^2+2}$ & $\Q_{a^2+2}$ & $\Q_{a^2+2}$ & $\Q_{a^2+2}$ & $\Q_{a^2+2}$ & $\Q_{a^2+2}$ \\
$2(a^2-2)$ & $\Q_{a^2-2}$ & $\Q_{a^2-2}$ & $(a^2 -2, 2a^3 - 4a)$ & $\Q_{a^2-2}$ & $\bigstar$ & $\Q_{a^2-2}$ & $\R$ & $\bigstar$ & $\Q_{a^2-2}$ & $\Q_{a^2-2}$ & $\R$ & $\Q_{a^2-2}$ & $\R$ & $\Q_{a^2-2}$ & $\R$ & $\bigstar$ \\
$2(a^2+2)$ & $\Q_{a^2+2}$ & $\Q_{a^2+2}$ & $\Q_{a^2+2}$ & $\Q_{a^2+2}$ & $\Q_{a^2+2}$ & $\Q_{a^2+2}$ & $\Q_{a^2+2}$ & $\Q_{a^2+2}$ & $\Q_{a^2+2}$ & $\Q_{a^2+2}$ & $\Q_{a^2+2}$ & $\Q_{a^2+2}$ & $\Q_{a^2+2}$ & $\Q_{a^2+2}$ & $\Q_{a^2+2}$ & $\Q_{a^2+2}$ \\
$(a^2-2)(a^2+2)$ & $\Q_{a^2+2}$ & $\Q_{a^2+2}$ & $\Q_{a^2+2}$ & $\Q_{a^2+2}$ & $\Q_{a^2+2}$ & $\Q_{a^2+2}$ & $\Q_{a^2+2}$ & $\Q_{a^2+2}$ & $\Q_{a^2+2}$ & $\Q_{a^2+2}$ & $\Q_{a^2+2}$ & $\Q_{a^2+2}$ & $\Q_{a^2+2}$ & $\Q_{a^2+2}$ & $\Q_{a^2+2}$ & $\Q_{a^2+2}$ \\
$2(a^2-2)(a^2+2)$ & $\Q_{a^2+2}$ & $\Q_{a^2+2}$ & $\Q_{a^2+2}$ & $\Q_{a^2+2}$ & $\Q_{a^2+2}$ & $\Q_{a^2+2}$ & $\Q_{a^2+2}$ & $\Q_{a^2+2}$ & $\Q_{a^2+2}$ & $\Q_{a^2+2}$ & $\Q_{a^2+2}$ & $\Q_{a^2+2}$ & $\Q_{a^2+2}$ & $\Q_{a^2+2}$ & $\Q_{a^2+2}$ & $\Q_{a^2+2}$ \\
$-1$ & {\fontsize{50}{50}\selectfont $\R$} & & & & & & & & $\R$ & $\bigstar$ & $\Q_{a^2-2}$ & $(-a^2 - 2,0)$ & $\Q_{a^2-2}$ & $\R$ & $\Q_{a^2-2}$ & $\Q_{a^2-2}$ \\
$-2$ & & & & & & & & & $(-a^2,2a)$ & $\R$ & $\Q_{a^2-2}$ & $\R$ & $\Q_{a^2-2}$ & $\bigstar$ & $\Q_{a^2-2}$ & $\Q_{a^2-2}$ \\
$-(a^2 - 2)$ & & & & & & & & & $\Q_{a^2-2}$ & $\Q_{a^2-2}$ & $(-a^2 + 2,0)$ & $\Q_{a^2-2}$ & $\R$ & $\Q_{a^2-2}$ & $\Q_{a^2-2}$ & $\Q_{a^2-2}$ \\
$-(a^2 + 2)$ & & & & & & & & & $\Q_{a^2+2}$ & $\Q_{a^2+2}$ & $\Q_{a^2+2}$ & $\Q_{a^2+2}$ & $\Q_{a^2+2}$ & $\Q_{a^2+2}$ & $\Q_{a^2+2}$ & $\Q_{a^2+2}$ \\
$-2(a^2-2)$ & & & & & & & & & $\Q_{a^2-2}$ & $\Q_{a^2-2}$ & $\R$ & $\Q_{a^2-2}$ & $\bigstar$ & $\Q_{a^2-2}$ & $\left(\dfrac{-a^4 + 4}{a^2}, \dfrac{-2a^4 + 8}{a^3}\right)$ & $\R$ \\
$-2(a^2+2)$ & & & & & & & & & $\Q_{a^2+2}$ & $\Q_{a^2+2}$ & $\Q_{a^2+2}$ & $\Q_{a^2+2}$ & $\Q_{a^2+2}$ & $\Q_{a^2+2}$ & $\Q_{a^2+2}$ & $\Q_{a^2+2}$ \\
$-(a^2-2)(a^2+2)$ & & & & & & & & & $\Q_{a^2+2}$ & $\Q_{a^2+2}$ & $\Q_{a^2+2}$ & $\Q_{a^2+2}$ & $\Q_{a^2+2}$ & $\Q_{a^2+2}$ & $\Q_{a^2+2}$ & $\Q_{a^2+2}$ \\
$-2(a^2-2)(a^2+2)$ & & & & & & & & & $\Q_{a^2+2}$ & $\Q_{a^2+2}$ & $\Q_{a^2+2}$ & $\Q_{a^2+2}$ & $\Q_{a^2+2}$ & $\Q_{a^2+2}$ & $\Q_{a^2+2}$ & $\Q_{a^2+2}$
\end{tblr}}}
\end{table}

\clearpage
\bibliographystyle{alpha}

\begin{thebibliography}{BCP97}

\bibitem[BCP97]{bosma1997magma}
Wieb Bosma, John Cannon, and Catherine Playoust.
\newblock The {M}agma algebra system I: The user language.
\newblock {\em Journal of Symbolic Computation}, 24(3-4):235--265, 1997.

\bibitem[Bom90]{bombieri1990mordell}
Enrico Bombieri.
\newblock The {M}ordell conjecture revisited.
\newblock {\em Annali della Scuola Normale Superiore di Pisa-Classe di Scienze}, 17(4):615--640, 1990.

\bibitem[BR22]{Barrios_2022}
Alexander~J. Barrios and Manami Roy.
\newblock Local data of rational elliptic curves with nontrivial torsion.
\newblock {\em Pacific Journal of Mathematics}, 318(1):1–42, August 2022.

\bibitem[Cas68]{cassels1968theorem}
JWS Cassels.
\newblock On a theorem of {D}em'janenko.
\newblock {\em Journal of the London Mathematical Society}, 1(1):61--66, 1968.

\bibitem[Cha41]{chabauty1941points}
Claude Chabauty.
\newblock Sur les points rationnels des courbes alg{\'e}briques de genre sup{\'e}rieura l’unit{\'e}.
\newblock {\em CR Acad. Sci. Paris}, 212(882-885):1, 1941.

\bibitem[Col85]{coleman1985}
Robert~F. Coleman.
\newblock {Effective Chabauty}.
\newblock {\em Duke Mathematical Journal}, 52(3):765 -- 770, 1985.

\bibitem[Dem68]{dem1966rational}
Vadim~Andreevich Dem'yanenko.
\newblock Rational points of a class of algebraic curves.
\newblock {\em Izvestiya Rossiiskoi Akademii Nauk. Seriya Matematicheskaya}, 30(6):1373--1396, 1968.

\bibitem[Fal83]{faltings1983endlichkeitssatze}
Gerd Faltings.
\newblock Endlichkeitss{\"a}tze f{\"u}r abelsche variet{\"a}ten {\"u}ber zahlk{\"o}rpern.
\newblock {\em Inventiones mathematicae}, 73:349--366, 1983.

\bibitem[GK05]{girard2005computation}
Martine Girard and Leopoldo Kulesz.
\newblock Computation of sets of {R}ational {P}oints of {G}enus-3 {C}urves via the {D}em'janenko--{M}anin {M}ethod.
\newblock {\em LMS Journal of Computation and Mathematics}, 8:267--300, 2005.

\bibitem[HS13]{hindry2013diophantine}
Marc Hindry and Joseph~H Silverman.
\newblock {\em Diophantine {G}eometry: {A}n {I}ntroduction}, volume 201.
\newblock Springer Science \& Business Media, 2013.

\bibitem[KD23]{kellock2023rootnumbersparityphenomena}
Lilybelle~Cowland Kellock and Vladimir Dokchitser.
\newblock Root numbers and parity phenomena, 2023.

\bibitem[Kul99]{kulesz1999application}
Leopoldo Kulesz.
\newblock Application de la méthode de {D}em'janenko--{M}anin à certaines familles de courbes de genre 2 et 3.
\newblock {\em Journal of Number Theory}, 76(1):130--146, 1999.

\bibitem[Man69]{manin1969p}
Ju~I Manin.
\newblock The p-torsion of elliptic curves is uniformly bounded.
\newblock {\em Selected Papers Of Yu I Manin}, 3:185, 1969.

\bibitem[Roh93]{rohrlich1993variation}
David~E Rohrlich.
\newblock Variation of the root number in families of elliptic curves.
\newblock {\em Compositio Mathematica}, 87(2):119--151, 1993.

\bibitem[Sil09]{silverman2009arithmetic}
Joseph~H Silverman.
\newblock {\em The {A}rithmetic of {E}lliptic {C}urves}, volume 106.
\newblock Springer, 2009.

\bibitem[SS58]{schinzel1958certaines}
Andrzej Schinzel and Wac{\l}aw Sierpi{\'n}ski.
\newblock Sur certaines hypotheses concernant les nombres premiers.
\newblock {\em Acta Arithmetica}, 4(3):185--208, 1958.

\bibitem[Voj91]{vojta1991siegel}
Paul Vojta.
\newblock Siegel's theorem in the compact case.
\newblock {\em Annals of Mathematics}, 133(3):509--548, 1991.

\end{thebibliography}

\end{document}